\documentclass[a4paper, reqno, oneside]{amsart}

\usepackage[latin1]{inputenc}
\usepackage{amsmath}
\usepackage{amsfonts}
\usepackage{amssymb}
\usepackage{graphicx}
\usepackage{enumerate}

\usepackage{color}
\usepackage{mathrsfs}
\usepackage{hyperref}

\long\def\symbolfootnote[#1]#2{\begingroup%
\def\thefootnote{\fnsymbol{footnote}}\footnote[#1]{#2}\endgroup}

\newtheorem{theorem}{Theorem}[section]

\theoremstyle{plain}

\newtheorem{case}{Case}

\newtheorem{subcase}{Subcase}[theorem]
\numberwithin{subcase}{case}

\newtheorem{claim}{Claim}

\newtheorem{lemma}[theorem]{Lemma}

\newtheorem{proposition}[theorem]{Proposition}

\begin{document}

\author{Karl Heuer}

\symbolfootnote[0]{\textcopyright 2019. This manuscript version is made available under the CC-BY-NC-ND 4.0 license \url{http://creativecommons.org/licenses/by-nc-nd/4.0/}}

\title[]{A sufficient local degree condition for Hamiltonicity in locally finite claw-free graphs}

\begin{abstract}

Among the well-known sufficient degree conditions for the Hamiltonicity of a finite graph, the condition of Asratian and Khachatrian is the weakest and thus gives the strongest result.
Diestel conjectured that it should extend to locally finite infinite graphs~$G$, in that the same condition implies that the Freudenthal compactification of $G$ contains a circle through all its vertices and ends.
We prove Diestel's conjecture for claw-free graphs.
\end{abstract}

\maketitle

\section{Introduction}

Problems concerning the existence of Hamilton cycles in finite graphs are studied quite a lot, but to decide whether a given finite graph is Hamiltonian is difficult.
Nevertheless, or even because of that, many sufficient or necessary conditions for Hamiltonicity have been found which are often easy to handle.
One common class of sufficient conditions are degree conditions.
An early result in this area is the following theorem of Dirac (1952).

\begin{theorem}\label{dirac-theorem}\cite[Thm.\ 3]{dirac}
Every finite graph with ${n \geq 3}$ vertices and minimum degree at least ${n / 2}$ is Hamiltonian.
\end{theorem}

The next result generalizes the theorem of Dirac and is due to Ore (1960).

\begin{theorem}\label{ore-theorem}\cite[Thm.\ 2]{ore}.
Let $G$ be a finite graph with ${n \geq 3}$ vertices. If \linebreak ${d(u) + d(v) \geq n}$ for any two non-adjacent vertices $u$ and $v$ of $G$, then $G$ is Hamiltonian.
\end{theorem}

Both of these theorems state sufficient conditions for Hamiltonicity which involve the total number of vertices in the given graph.
So we could say that both conditions do not have a local form.
Furthermore, these conditions imply that the considered graphs have diameter at most $2$.
In contrast to this, the next theorem, which is due to Asratian and Khachatrian \cite{asra-good}, generalizes both theorems above and allows graphs of arbitrary diameter.
In order to state the theorem, we need the following local property of a graph:

\[ d(u) + d(w) \geq |N(u) \cup N(v) \cup N(w)|\textit{ for every induced path } uvw. \tag{$\ast$} \]
\\
The theorem of Asratian and Khachatrian can now be formulated as follows:

\begin{theorem}\label{Asra-fin}\cite[Thm.\ 1]{asra-good}
Every finite connected graph which satisfies $(\ast)$ and has at least three vertices is Hamiltonian.
\end{theorem}

The vast majority of Hamiltonicity results deal only with finite graphs, since it is not clear what a Hamilton cycle in an infinite graph should be.
We follow the topological approach initiated by Diestel and K\"{u}hn \cite{diestel_kuehn_1, diestel_kuehn_2, diestel_kuehn_TST} and further outlined in \cite{diestel_buch, diestel_arx}, which solves this problem in a reasonable way by using as infinite cycles of a graph $G$ the circles in its Freudenthal compactification $|G|$.
Then circles which use infinitely many vertices of $G$ are possibly to exist.
Now the notion of a Hamiltonicity extends in an obvious way:
Call a locally finite connected graph $G$ \textit{Hamiltonian} if there is a circle in $|G|$ that contains all vertices of $G$.

Some Hamiltonicity results for finite graphs have already been generalized to locally finite graphs using this notion but not all of them are complete generalizations.
Theorems that involve local conditions as in Theorem~ \ref{Asra-fin} are more likely to generalize to locally finite graphs since they are still well-defined for infinite graphs and might allow compactness arguments.
For results in this field, see \cite{brewster-funk, bruhn-HC, agelos-HC, Ha_Leh_Po, heuer_Inf_ObSu, lehner-HC}.

This paper deals with a conjecture of Diestel \cite[Conj.\ 4.13]{diestel_arx} about Hamiltonicity which says that Theorem~\ref{Asra-fin} can be generalized to locally finite graphs.
The main result of this paper is the following theorem, which shows that the conjecture of Diestel holds for claw-free graphs where we call a graph \textit{claw-free} if it does not contain the claw, i.e., the graph $K_{1, 3}$, as an induced subgraph.

\begin{theorem}\label{Asra-loc-fin}
Every locally finite, connected, claw-free graph which satisfies $(\ast)$ and has at least three vertices is Hamiltonian.
\end{theorem}

The rest of this paper is structured in the following way.
In Section~2 we recall some basic definitions and introduce some notation we shall need in this paper.
Section~3 contains some facts and lemmas which are needed in the proof of our main result.
In the last section, Section~4, we consider locally finite graphs which satisfy condition~$(\ast)$.
There we give two infinite classes of examples of locally finite graphs satisfying~$(\ast)$.
In one class, all members are claw-free, while all elements of the other class have claws as induced subgraphs.
The rest of Section~4 deals with the proof of Theorem~\ref{Asra-loc-fin}.
At the very end of the paper we discuss where we need the assumption of being claw-free for the proof of our main theorem.

\section{Basic definitions and notation}

In general, we follow the graph-theoretic notation of \cite{diestel_buch} in this paper.
For basic graph-theoretic facts, we refer the reader also to~\cite{diestel_buch}.
Beside finite graph theory, a topologically approach to infinite locally finite graphs is covered in \cite[Ch.\ 8.5]{diestel_buch}.
For a survey in this field, we refer to \cite{diestel_arx}.

All graphs considered in this paper are undirected and simple.
Furthermore, a graph is not assumed to be finite.
Now we fix an arbitrary graph $G = (V, E)$ for this section.

The graph $G$ is called \textit{locally finite} if every vertex of $G$ has only finitely many neighbours.

For a vertex set $X$ of $G$, we denote by $G[X]$ the induced subgraph graph of $G$ whose vertex set is $X$.
We write $G-X$ for the graph $G[V \setminus X]$, but for singleton sets, we omit the set brackets and write just $G-v$ instead of $G-\lbrace v \rbrace$ where $v \in V$.
We denote the cut which consists of all edges of $G$ that have one endvertex in $X$ and the other endvertex in $V \setminus X$ by $\delta(X)$.

Let $C$ be a cycle of $G$ and $u$ be a vertex of $C$.
Then we write $u^+$ and $u^-$ for the neighbour of $u$ in $C$ in positive and negative, respectively, direction of $C$ given a fixed orientation of $C$.
We will not mention that we fix an orientation for the considered cycle using this notation.
We implicitly fix an arbitrary orientation of the cycle.

Let $P$ be a path in $G$ and $T$ a tree in $G$.
We write $\mathring{P}$ for the subpath of $P$ which is obtained from $P$ by removing the endvertices of $P$.
If $s$ and $t$ are vertices of $T$, we write $sTt$ for the unique $s$--$t$ path in $T$.
Note that this covers also the case where $T$ is a path.
If $P_v = v_0 \ldots v_n$ and $P_w = w_0 \ldots w_k$ are paths in $G$ with $n, k \in \mathbb{N}$ where $v_n$ and $w_0$ may be equal but apart from that these paths are disjoint and the vertices $v_n, w_0$ are the only vertices of $P_v$ and $P_w$ which lie in $T$, then we write $v_0 \ldots v_nTw_0 \ldots w_k$ for the path with vertex set ${V(P_v) \cup V(v_nTw_0) \cup V(P_w)}$ and edge set ${E(P_v) \cup E(v_nTw_0) \cup E(P_w)}$.

For a vertex set $X \subseteq V$ and an integer $k \geq 1$, we denote by $N_k(X)$ the set of vertices in $G$ that have distance at least $1$ and at most $k$ to $X$ in $G$.
For $k=1$ we just write $N(X)$ instead of $N_1(X)$, which denotes the usual neighbourhood of $X$ in $G$.
For a singleton set $\lbrace v \rbrace \subseteq V$, we omit the set brackets and write just $N_k(v)$ and $N(v)$ instead of $N_k(\lbrace v \rbrace)$ and $N(\lbrace v \rbrace)$, respectively.
Given a subgraph $H$ of $G$, we just write $N_k(H)$ and $N(H)$ instead of $N_k(V(H))$ and $N(V(H))$, respectively.

We denote the graph $K_{1, 3}$ also as \textit{claw}.
The graph $G$ is called \textit{claw-free} if it does not contain the claw as an induced subgraph.

A one-way infinite path in $G$ is called a \textit{ray} of $G$ and a two-way infinite path in $G$ is called a \textit{double ray} of $G$.
An equivalence relation can be defined on the set of all rays of $G$ by saying that two rays in $G$ are equivalent if they cannot be separated by finitely many vertices.
It is easy to check that this relation really defines an equivalence relation.
The corresponding equivalence classes of this relation are called the \textit{ends} of $G$.

For the rest of this section, we assume $G$ to be locally finite and connected.
A topology can be defined on $G$ together with its ends to obtain the topological space $|G|$.
For a precise definition of $|G|$, see \cite[Ch.\ 8.5]{diestel_buch}.
An important fact about $|G|$ is that every ray of $G$ converges to the end of $G$ it is contained in.

Apart from the definition of $|G|$ as in \cite[Ch.\ 8.5]{diestel_buch}, there is an equivalent way of defining the topological space $|G|$:
Endow $G$ with the topology of a $1$-complex (also called CW complex of dimension $1$) and then build the Freudenthal compactification of $G$. This connection was examined in \cite{freud-equi}.

For a point set $X$ in $|G|$, we denote its closure in $|G|$ by $\overline{X}$.

We define a \textit{circle} in $|G|$ as the image of a homeomorphism which maps from the unit circle $S^1$ in $\mathbb{R}^2$ to $|G|$.
The graph $G$ is called \textit{Hamiltonian} if there is a circle in $|G|$ containing all vertices of $G$.
Such a circle is called a \textit{Hamilton circle} of $G$.
For finite $G$, this coincides with the usual meaning, namely the existence of a cycle in $G$ which contains all vertices of $G$.
Such cycles are called \textit{Hamilton cycles} of $G$.

\section{Toolkit}

In this section we collect some lemmas which we shall need later for the proof of the main result.
The proof of each statement of this section can be found in \cite[Section 3]{heuer_Inf_ObSu}.
We begin with two basic facts about minimal vertex separators in claw-free graphs.

\begin{proposition}\label{2 comp}
Let $G$ be a connected claw-free graph and $S$ be a minimal vertex separator in $G$. Then $G-S$ has exactly two components.
\end{proposition}

The following lemma together with Proposition~\ref{2 comp} are essential for the constructive proof of Lemma~\ref{Asra-cut-1}.

\begin{lemma}\label{complete}
Let $G$ be a connected claw-free graph and $S$ be a minimal vertex separator in $G$.
For every vertex $s \in S$ and every component $K$ of $G-S$, the graph $G[N(s)\cap V(K)]$ is complete.
\end{lemma}

We proceed with a structural lemma on infinite, locally finite, connected, claw-free graphs (see Figure~1).

\begin{lemma}\label{struct_toll}\cite[Lemma 3.10]{heuer_Inf_ObSu}
Let $G$ be an infinite, locally finite, connected, claw-free graph and ${X}$ be a finite vertex set of $G$ such that $G[X]$ is connected.
Furthermore, let $\mathscr{S} \subseteq V(G)$ be a finite minimal vertex set such that $\mathscr{S} \cap X = \emptyset$ and every ray starting in $X$ has to meet $\mathscr{S}$.
Then the following holds:
\begin{enumerate}[\normalfont(i)]

\item $G-\mathscr{S}$ has $k \geq 1$ infinite components $K_1, \ldots, K_k$ and the set $\mathscr{S}$ is the disjoint union of minimal vertex separators $S_1, \ldots, S_k$ in $G$ such that for every $i$ with $1 \leq i \leq k$ each vertex in $S_i$ has a neighbour in $K_j$ if and only if $j=i$.
\item $G-\mathscr{S}$ has precisely one finite component $K_0$. This component contains all vertices of $X$ and every vertex of $\mathscr{S}$ has a neighbour in $K_0$.
\end{enumerate}
\end{lemma}

\begin{figure}[htbp]
\centering
\includegraphics{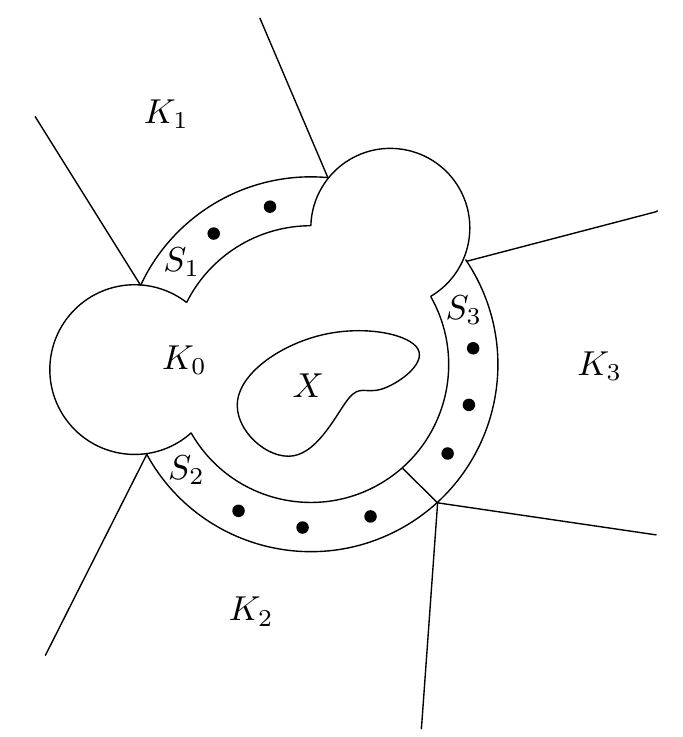}
\caption{The structure described in Lemma~\ref{struct_toll} for $k=3$}
\label{lemma_3_3_k_3}
\end{figure}

To prove that an infinite, locally finite, connected graph $G$ is Hamiltonian, we use the following lemma.
A sequence of cycles of $G$ and a set of vertex sets which fulfill five conditions are needed to be able to apply this lemma.
While a Hamilton circle is built from the sequence of cycles as a limit object, the vertex sets help to control
this process by witnessing that the limit does really become a circle.
Since this lemma is our key to prove Hamiltonicity, let us consider the idea of the lemma more carefully before we state it.
We define a limit object from a sequence of cycles of $G$ by saying that a vertex or an edge of $G$ is contained in the limit if it lies in all but finitely many cycles of the sequence.
Of course we must be able to tell for each vertex and for each edge of $G$ whether it is in the limit or not.
The conditions $(\text{i})$ and $(\text{iv})$ of the lemma ensure that we can do this.
Furthermore, condition $(\text{i})$ forces every vertex of $G$ to be in the limit, which is necessary for the limit to be a Hamilton circle.
Ensuring that the limit object becomes a circle consists of two parts.
One thing is to guarantee that the limit is topologically connected and that the degree of each vertex is two in the limit.
Condition $(\text{iv})$ and the definition of the limit take care of this such that both of these properties can easily be verified.
The problematic part is to ensure that all ends have degree $2$ in the limit object and not higher.
Both parts together are equivalent to the limit object being a circle by a result of Bruhn and Stein~\cite[Prop.~3]{circle}.
In fact, without further conditions, there might be ends with degree higher than $2$ in the limit.
To prevent this problem, we use the conditions $(\text{ii})$, $(\text{iii})$ and $(\text{v})$.
They guarantee the existence of a sequence of finite cuts for each end such that each sequence converges to its corresponding end and the limit object meets each of the cuts precisely twice.
This prevents ends having a degree higher that $2$ in the limit object.
Due to the five conditions of the lemma, it is not easy to find cycles and vertex sets which can be used for the application of the lemma.
Indeed, the main work for the proof of Theorem~\ref{Asra-loc-fin} is to construct cycles and vertex sets which fulfill the required conditions.
Especially the structure of a graph as described in Lemma~\ref{struct_toll} will help us in the construction.

\begin{lemma}\label{HC-extract}\cite[Lemma 3.11]{heuer_Inf_ObSu}
Let $G$ be an infinite, locally finite, connected graph and $(C_i)_{i \in \mathbb{N}}$ be a sequence of cycles of $G$.
Then $G$ is Hamiltonian if there exists an integer $k_i \geq 1$ for every $i \geq 1$ and vertex sets $M^i_j \subseteq V(G)$ for every $i \geq 1$ and $j$ with $1 \leq j \leq k_i$ such that the following is true:
\textnormal{
\begin{enumerate}[\normalfont(i)]
\item \textit{For every vertex $v$ of $G$, there exists an integer $j \geq 0$ such that $v \in V(C_i)$ holds for every $i \geq j$.}
\item \textit{For every $i \geq 1$ and $j$ with $1 \leq j \leq k_i$, the cut $\delta(M^i_j)$ is finite.}
\item \textit{For every end $\omega$ of $G$, there is a function $f : \mathbb{N} \setminus \lbrace 0 \rbrace \longrightarrow \mathbb{N}$ such that the inclusion ${M^{j}_{f(j)} \subseteq M^i_{f(i)}}$ holds for all integers $i, j$ with $1 \leq i \leq j$ and the equation ${M_{\omega}:= \bigcap^{\infty}_{i=1} \overline{M^i_{f(i)}} = \lbrace \omega \rbrace}$ is true.}
\item \textit{$E(C_i) \cap E(C_j) \subseteq E(C_{j+1})$ holds for all integers $i$ and $j$ with $0 \leq i < j$.}
\item \textit{The equations $E(C_i) \cap \delta(M^p_j) = E(C_p) \cap \delta(M^p_j)$ and $|E(C_i) \cap \delta(M^p_j)| = 2$ hold for each triple $(i, p, j)$ which satisfies $1 \leq p \leq i$ and $1 \leq j \leq k_p$.}
\end{enumerate}
}
\end{lemma}

\section{A local degree condition}

This section deals with locally finite graphs satisfying the following degree condition:

\[ d(u) + d(w) \geq |N(u) \cup N(v) \cup N(w)|\textit{ for every induced path } uvw. \tag{$\ast$} \]
\\
Asratian and Khachatrian gave an example of a class of finite graphs which satisfy this condition and have arbitrarily high diameter (see \cite{asra-good}).
Seizing the idea of their examples we show in the first part of this section that it is easy to construct locally finite infinite graphs satisfying $(\ast)$, even claw-free ones.
In order to state their examples, we use the notion of lexicographic products of graphs.

Let $G$ and $H$ be two graphs.
Then the \textit{lexicographic product} $G \circ H$ of $G$ and $H$ is the graph on the vertex set ${V(G \circ H) = V(G) \times V(H)}$ where two vertices $(u_1, h_1)$ and $(u_2, h_2)$ are adjacent if and only if either $u_1u_2 \in E(G)$ or $u_1 = u_2$ and $h_1h_2 \in E(H)$.

Now we can state the examples of Asratian and Khachatrian.
Let $G_{q,n} = C_q \circ K_n$ where $C_q$ is the cycle of length $q \geq 3$ and $K_n$ is the complete graph on $n \geq 2$ vertices (see Figure~2).
The definition of $G_{q,n}$ ensures that the diameter of $G_{q,n}$ is $\lfloor {q/2} \rfloor$.
To see that $G_{q,n}$ satisfies $(\ast)$ for every $q \geq 3$ and every $n \geq 2$, note that the equations
$${d(u) + d(w) = 6n-2} \;  \textnormal{ and } \;  |N(u) \cup N(v) \cup N(w)| = 5n$$
hold for each induced path $uvw$ of $G_{q,n}$.
Since we assume $n \geq 2$, the graph $G_{q,n}$ satisfies $(\ast)$.

We can obtain infinite locally finite graphs satisfying $(\ast)$ in the same way.
Let $G_{\mathbb{Z}, n} = D \circ K_n$ where $D$ is a double ray and $K_n$ is the complete graph on $n \geq 2$ vertices (see Figure~2).
By the same equations as above, the graph $G_{\mathbb{Z}, n}$ \linebreak satisfies $(\ast)$ for every $n \geq 2$.

\begin{figure}[htbp]
\centering
\includegraphics[width=7.5cm]{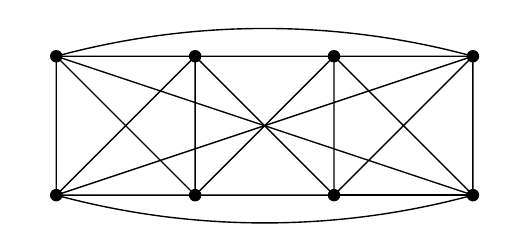}
\includegraphics[width=10cm]{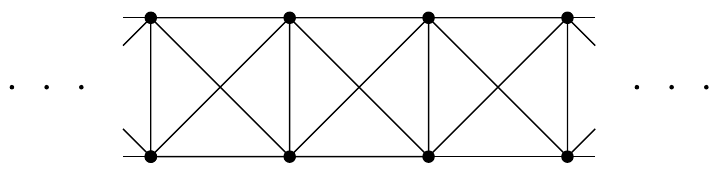}
\caption{The graph $G_{4,2}$ and the graph $G_{\mathbb{Z}, 2}$ below.}
\label{asra-beispiel}
\end{figure}

Note that the graphs $G_{q,n}$ and $G_{\mathbb{Z}, n}$ are claw-free.
In general, being claw-free does not follow from $(\ast)$.
To see this, we give examples of such graphs, whose structure is very similar to $G_{q,n}$ and $G_{\mathbb{Z}, n}$.
We begin with the definition of the graph $H_{2q, n}$.
Let ${V(H_{2q, n}) = A \times V(K_{1,3}) \cup B \times V(K_{n})}$ where ${A, B \subseteq V(C_{2q})}$ are the partition classes of a bipartition of the cycle $C_{2q}$ of length $2q$.
Two vertices $(a, v)$ and $(b, w)$ of $H_{2q, n}$ are adjacent if ${ab \in E(C_{2q})}$ or if ${a = b \in A}$ and ${vw \in E(K_{1,3})}$ or if ${a = b \in B}$ and ${vw \in E(K_n)}$.
The graph $H_{\mathbb{Z}, n}$ is defined analogously where the cycle $C_{2q}$ is replaced by a double ray $D$.
The argumentation for checking that $H_{2q, n}$ and $H_{\mathbb{Z}, n}$ satisfy $(\ast)$ for ${q \geq 2}$ and ${n \geq 6}$ are the same.
So we look only at $H_{2q, n}$.
Let $(a, u)(b, v)(c, w)$ be an induced path in $H_{2q, n}$ or $H_{\mathbb{Z}, n}$.
We distinguish four cases. If ${a = b = c}$ holds, then $uvw$ is an induced path in $K_{1, 3}$ and so we get
$${d((a, u)) + d((c, w)) = 4n + 2 \geq 2n + 4 = |N((a, u)) \cup N((b, v)) \cup N((c, w))|}.$$
If $a = c \neq b$, then $u$ and $w$ are nonadjacent vertices of $K_{1,3}$ and $v$ lies in $K_n$.
This gives
$$d((a, u)) + d((c, w)) = 4n + 2 \geq 2n + 8 = |N((a, u)) \cup N((b, v)) \cup N((c, w))|$$
since $n \geq 3$. If $a$, $b$ and $c$ are pairwise distinct and $a, c \in B$, then we obtain
$$d((a, u)) + d((c, w)) = 2n + 14 \geq 2n + 12 \geq |N((a, u)) \cup N((b, v)) \cup N((c, w))|.$$
Note that for $q \geq 3$ the second inequality becomes an equality.
In the last case, the vertices $a$, $b$ and $c$ are pairwise distinct and $a, c \in A$.
We get the inequality chain
$$d((a, u)) + d((c, w)) \geq 4n + 2 \geq 3n + 8 = |N((a, u)) \cup N((b, v)) \cup N((c, w))|$$
since $n \geq 6$.
Note here that the first inequality becomes an equality if $u$ and $v$ are both vertices of degree $1$ in $K_{1, 3}$.
These examples show that there are finite and infinite locally finite graphs which satisfy $(\ast)$, have arbitrary high diameter and contain induced claws.
The graphs $H_{\mathbb{Z}, n}$ show that even infinitely many induced claws can be contained in graphs satisfying $(\ast)$.

Next we state a basic lemma about graphs having the property $(\ast)$.

\begin{lemma}\label{ungl-kette}\cite{asra-good}
Let $G$ be a graph which satisfies $(\ast)$.
Then the following is true for every induced path $uvw$ of $G$:
\[ |N(u) \cap N(w)| \geq |N(v) \setminus (N(u) \cup N(w))| \geq 2. \]
\end{lemma}

\begin{proof}
Let $uvw$ be an induced path of $G$. We get the following inequality chain:
\begin{align*}
|N(u) \cap N(w)| &= d(u) + d(w) - |N(u) \cup N(w)| \\
				 &\geq |N(u) \cup N(v) \cup N(w)| - |N(u) \cup N(w)| \\
				 &= |N(v) \setminus (N(u) \cup N(w))| \\
				 &\geq |\lbrace u, w \rbrace| = 2
\end{align*}
All equalities in this chain are obvious.
The first inequality is valid due to $(\ast)$.
The second inequality holds because $uvw$ is an induced path, which means that $u$ and $w$ are not adjacent.
\end{proof}

The proof of Theorem~\ref{Asra-fin} which Asratian and Khachatrian presented in \cite{asra-good} is basically the same as the one for the following lemma. For the sake of completeness we state the proof here.

\begin{lemma}\label{Asra-enlarge}
Let $G$ be a locally finite graph satisfying $(\ast)$, $C$ be a cycle of $G$ and $v$ be a vertex in $N(C)$.
Then there exists a vertex $u \in N(v) \cap V(C)$ such that one of the following is true:
\textnormal{
\begin{enumerate}[\normalfont(i)]
\item \textit{The vertex $v$ is adjacent with $u^+$.}
\item \textit{There is a vertex $x \in N(v) - V(C)$ such that $x$ is adjacent with $u^+$.}
\item \textit{There exists a vertex $y \neq u$ in $N(v) \cap V(C)$ such that $u^+ \neq y$ and $u^+$ and $y^+$ are adjacent.}
\end{enumerate}
}
\end{lemma}

\begin{proof}
Since cycles are finite, we get that the set $N(v) \cap V(C)$ is finite.
Let ${N(v) \cap V(C) = \lbrace u_1, \ldots, u_n \rbrace}$.
We know that $n \geq 1$ holds since $v$ lies in the neighbourhood of $C$.
We may assume that $vu^+_i \notin E(G)$ is true for all $i$ with $1 \leq i \leq n$ because otherwise statement~(i) of the lemma would hold and we would be done.
This assumption implies ${\lbrace u_1, \ldots, u_n \rbrace \cap \lbrace u^+_1, \ldots, u^+_n \rbrace = \emptyset}$.
So we may assume further that ${\lbrace u^+_1, \ldots, u^+_n \rbrace}$ is an independent set because otherwise statement~(iii) of the lemma would be true and the proof complete.

For every ${i \in \lbrace 1, \ldots, n \rbrace}$ let
$${I_i = N(v) \cap N(u^+_i)}  \;  \textnormal{ and } \;   {M_i = N(u_i) \setminus (N(v) \cup N(u^+_i))}.$$
By our assumptions, we know that $vu_iu^+_i$ is an induced path for every ${i \in \lbrace 1, \ldots, n \rbrace}$.
So Lemma~\ref{ungl-kette} implies that ${|I_i| \geq |M_i|}$ and ${|I_i| \geq 2}$ are true for each ${i \in \lbrace 1, \ldots, n \rbrace}$.

Now we will show that there is a vertex ${x \in I_j \setminus V(C)}$ for some ${j \in \lbrace 1, \ldots, n \rbrace}$, which implies statement~(ii) of the lemma.
For this, we construct a function $f(k)$ and sequences $(Z^k_i)$, $(Y^k_i)$ for each ${i \in \lbrace 1, \ldots, n \rbrace}$ where the relations ${1 \leq f(k) \leq n}$ and ${Z^k_i \subseteq \lbrace u^+_1, \ldots, u^+_n \rbrace \cup \lbrace v \rbrace}$ and ${Y^k_i \subseteq \lbrace u_1, \ldots, u_n \rbrace}$ are always valid.
We begin by setting ${f(1) = 1}$, ${Z^1_i = \lbrace v, u^+_i \rbrace}$ and ${Y^1_i = \lbrace u_i \rbrace}$ for each ${i \in \lbrace 1, \ldots, n \rbrace}$.
Now assume we have already defined $f(i)$ for every ${i \in \lbrace 1, \ldots, k \rbrace}$ and both sequences up to length $k$ for $k \geq 1$. If ${(I_{f(k)} \setminus Y^k_{f(k)}) \subseteq V(C)}$ and ${I_{f(k)} \setminus Y^k_{f(k)}}$ is nonempty, take a vertex $w$ of ${I_{f(k)} \setminus Y^k_{f(k)}}$.
In this case, we know that ${w = u_r}$ for some $r$ with ${1 \leq r \leq n}$ and proceed with the construction as follows for every ${i \in \lbrace 1, \ldots, n \rbrace}$:
\[
f(k+1) = r,
\]
\[
Y^{k+1}_i = 
\begin{cases}
Y^{k}_{f(k)} \cup \lbrace u_r \rbrace & \mbox{if } i = f(k) \\
Y^{k}_r 								 & \mbox{if } i = r \\
Y^{k}_i 								 & \mbox{otherwise},
\end{cases}\]
\[
Z^{k+1}_i = 
\begin{cases}
Z^{k}_{f(k)} 					  & \mbox{if } i = f(k) \\
Z^{k}_r \cup \lbrace u^+_{f(k)} \rbrace & \mbox{if } i = r \\
Z^{k}_i 							  & \mbox{otherwise}.
\end{cases}
\]

Next we gather some facts about the sequences.

\setcounter{claim}{0}
\begin{claim}
If $f(i)$ is defined for every ${i \in \lbrace 1, \ldots, k \rbrace}$ and both sequences up to length $k$ for ${k \geq 1}$, then we get:
\textnormal{
\begin{enumerate}[\normalfont(a)]
\item \textit{The relations ${Z^k_i \subseteq M_i}$, ${Y^k_i \subseteq I_i}$ and ${|Z^k_i| \geq |Y^k_i|}$ hold for each ${i \in \lbrace 1, \ldots, n \rbrace}$}.
\item \textit{${|Z^k_{f(k)}| > |Y^k_{f(k)}|}$}.
\item \textit{${I_{f(k)} \setminus Y^k_{f(k)} \neq \emptyset}$}.
\item \textit{${|Y^k_{f(k-1)}| = |Y^{k-1}_{f(k-1)}| + 1}$ if ${k \geq 2}$}.
\end{enumerate}
}
\end{claim}
We prove Claim~1 by induction on $k$.
It is easily checked that for $k=1$ all statements are true.
Now assume all statements are true for some $k \geq 1$.
We show statement~(a) for $k+1$ first.
The relations ${Z^{k+1}_i \subseteq M_i}$ and ${Y^{k+1}_i \subseteq I_i}$ follow by the induction hypothesis and construction.
For $i \neq f(k)$, the relation ${|Z^{k+1}_i| \geq |Y^{k+1}_i|}$ follows using the induction hypothesis and by the definition of both sets.
If $i = f(k)$, the inequality is true because ${|Z^k_{f(k)}| > |Y^k_{f(k)}|}$ holds by the induction hypothesis.

Statement~(b) for $k+1$ follows by definition of both sets and by the inequality ${|Z^k_{f(k+1)}| \geq |Y^k_{f(k+1)}|}$ of the induction hypothesis.

To prove statement~(c) for $k+1$, note that ${|M_{f(k+1)}| \geq |Z^{k+1}_{f(k+1)}| > |Y^{k+1}_{f(k+1)}|}$ is true by statement~(a) and (b) for $k+1$.
Now the statement follows since the inequality ${|I_{f(k+1)}| \geq |M_{f(k+1)}|}$ holds as shown before.

Statement~(d) is obviously true for $k=2$ and follows for arbitrary ${k+1 > 2}$ from the definition of ${Y^{k+1}_{f(k)}}$.
This completes the proof of Claim~1.

Now we see that there is some $j$ and a vertex ${{x \in I_{f(j)} \setminus (Y^{j}_{f(j)} \cup V(C))}}$.
Otherwise, statement~(c) of Claim~1 implies that we do not stop constructing the sequences $(Z^k_i)$ and $(Y^k_i)$.
This yields a contradiction because then it follows from statement~(d) of Claim~1 and the pigeonhole principle that there exists some ${p \in \lbrace 1, \ldots, n \rbrace}$ such that ${|Y^{k}_{p}| \rightarrow \infty}$ for ${k \rightarrow \infty}$ although ${Y^{k}_{p} \subseteq \lbrace u_1, \ldots, u_n \rbrace}$.
\end{proof}

To each case of Lemma~\ref{Asra-enlarge} corresponds a cycle that contains $v$ and all vertices of $C$.
We get these cycles by taking $C$ and replacing some of its edges.
For case~(i), we replace the edge $uu^+$ by the path $uvu^+$.
For case~(ii), we take the path $uvxu^+$ instead of the edge $uu^+$.
In case~(iii), we delete the edges $uu^+, yy^+$ and add the edges $uv, vy, u^+y^+$ to obtain the desired cycle if $u^+$ and $y^+$ are adjacent.
We call each such resulting cycle an \textit{extension of} $C$, or more precise a (i)-, (ii)- \linebreak or (iii)-\textit{extension of} $C$ depending on from which of the three cases of the lemma we obtain the resulting cycle.
A vertex which is used for an extension of a cycle as $v$ above is called the \textit{target} of the extension (see Figure~3).
For a cycle $C$, we call a finite sequence of cycles $(C_i)$, where ${0 \leq i \leq n}$ for some ${n\in \mathbb{N}}$, an \textit{extension sequence of} $C$ if $C_0 = C$ and $C_i$ is an extension of $C_{i-1}$ for every $i$ satisfying ${1 \leq i \leq n}$.

Note that Theorem~\ref{Asra-fin} follows easily from Lemma~\ref{Asra-enlarge} together with Lemma~\ref{ungl-kette}.

\begin{figure}[htbp]
\centering
\includegraphics[width=3cm]{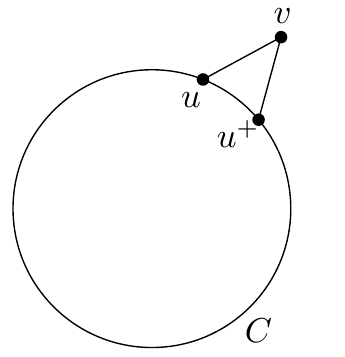}
\includegraphics[width=3cm]{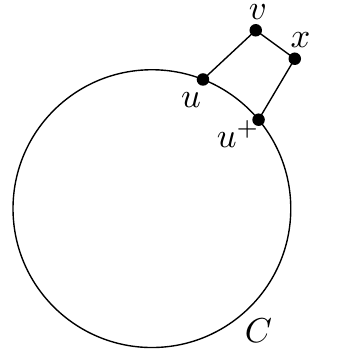}
\includegraphics[width=3cm]{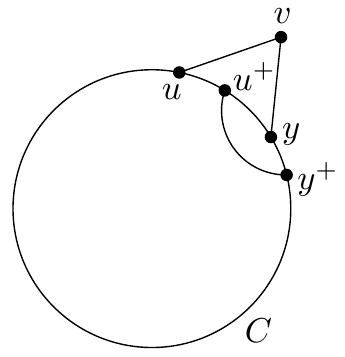}
\caption{(i)-, (ii)- and (iii)-extension of a cycle $C$ with target $v$.}
\label{extensions}
\end{figure}

Now we have nearly everything we need to prove Theorem~\ref{Asra-loc-fin}.
Before we can focus on the proof of the main theorem we state the following lemma, which is the key to make Lemma~\ref{HC-extract} applicable to the proof of the main theorem by helping us to define a suitable sequence of cycles together with vertex sets.
To prove the following lemma we make much use of extensions of cycles, which we obtain as in Lemma~\ref{Asra-enlarge}.

\begin{lemma}\label{Asra-cut-1}
Let $G=(V,E)$ be an infinite locally finite, connected, claw-free graph which satisfies $(\ast)$, $C$ be a cycle of $G$ with ${V(C) \setminus N(N(C)) \neq \emptyset}$ and ${\mathscr{S} \subseteq N(C)}$ be a minimal vertex set such that every ray starting in $C$ meets $\mathscr{S}$.
Furthermore, let $k$, $S_j$ and $K_j$ be analogously defined as in Lemma~\ref{struct_toll}.
Then there exists a cycle $C'$ with the properties:
\textnormal{
\begin{enumerate}[\normalfont(i)]
\item \textit{${V(K_0) \cup \mathscr{S} \cup N_3(\mathscr{S}) \subseteq V(C')}$.}
\item \textit{For every $j$ with ${1 \leq j \leq k}$, there are vertex sets ${M_j \subseteq V}$ such that the inclusions ${V(K_j) \setminus N(S_j) \subseteq M_j \subseteq V(K_j) \cup \mathscr{S} \cup N(\mathscr{S})}$ as well as the equation ${|E(C') \cap \delta(M_j)| = 2}$ are true.}
\item \textit{${E(C - N(N(C))) \subseteq E(C')}$ and for every edge ${e = uv \in E(C') \setminus E(C)}$ the inclusion ${\lbrace u, v \rbrace \subseteq (V \setminus V(C)) \cup N_2(N(C))}$ holds.}
\end{enumerate}
}
\end{lemma}

\begin{proof}
First we construct an extension sequence $(C_i)$ of $C$ where the vertex set of the last element in this sequence consists precisely of $V(K_0)$.
We begin by setting $C_0 = C$.
Note that ${V(C) \subseteq V(K_0)}$ is true by choice of $K_0$ and Lemma~\ref{struct_toll}.

Assume we have already constructed an extension sequence of $C$ of length $m+1$ for $m \geq 0$ such that ${V(C_m) \subseteq V(K_0)}$ holds.
If ${V(C_{m}) = V(K_0)}$ holds, we are done.
Otherwise ${V(C_{m}) \neq V(K_0)}$ is valid.
Then there exists a vertex ${v \in N(C_m) \cap V(K_0)}$.
Choose this vertex as the target of an extension $C'_m$ of $C_m$.
We can find such an extension by Lemma~\ref{Asra-enlarge}.
If $V(C'_m)$ is still a subset of $V(K_0)$, we set $C_{m+1} = C'_m$. Otherwise $C'_m$ is a (ii)-extension of $C_m$ and contains a vertex ${x \in \mathscr{S} \setminus V(C_m)}$.
Say ${x \in S_j}$ for some $j$ with ${1 \leq j \leq k}$.
We know that $C'_m-x$ lies entirely in one component of $G-S_j$ because $C'_m-x$ is connected and contains no vertex of the separator $S_j$.
Since $G$ is a connected claw-free graph and $S_j$ is a minimal vertex separator in $G$, we obtain by Lemma~\ref{complete} that the neighbourhood of $S_j$ in each component of $G-S_j$ induces a complete graph.
So $v$ and the other neighbour of $x$ in $C'_m$ are adjacent.
Hence, there exists a (i)-extension of $C_m$ with target $v$ which we choose as $C_{m+1}$. Since $K_0$ is finite by Lemma~\ref{struct_toll}, there has to be an integer $n$ such that ${V(C_n) = V(K_0)}$.

Now we construct a sequence $(\tilde{C}_{i})$ of cycles with ${0 \leq i \leq k}$ such that the following properties hold:

\begin{itemize}
\item ${V(K_0) \subseteq V(\tilde{C}_{i})}$ for every $i$ with ${0 \leq i \leq k}$.
\item For every $j$ with ${1 \leq j \leq k}$ the cycle $\tilde{C}_{j}$ contains precisely two vertices ${s^{\ell}_1, s^{\ell}_2}$ from $S_\ell$ and an $s^{\ell}_1$--$s^{\ell}_2$ path $P_\ell$ which satisfies ${V(\mathring{P}_\ell) \neq \emptyset}$ and ${V(\mathring{P}_\ell) \subseteq V(K_{\ell})}$ for every $\ell$ with ${1 \leq \ell \leq j}$.
\item $\tilde{C}_{j}$ contains no vertices from ${\bigcup^k_{p=j+1} (S_p \cup K_p)}$ for every $j$ with ${0 \leq j \leq k}$.
\end{itemize}
We start to build this sequence by setting ${\tilde{C}_{0} = C_n}$.
For every $j$ with ${1 \leq j \leq k}$ let $T_j$ be a finite tree in $K_j$ which contains all vertices from ${N_3(S_j) \cap V(K_j)}$.
Such trees exist because $K_j$ is connected and $N_3(S_j)$ is finite since $G$ is locally finite and $S_j$ is finite.
We use these trees not only in the construction of the sequence $(\tilde{C}_{i})$ but also afterwards.

Now assume we have already constructed such a sequence of length $m+1$ for $m \geq 0$. Let $D$ be an extension of $\tilde{C}_{m}$ with target ${s^{m+1}_1 \in S_{m+1}}$.
We distinguish two cases:

\setcounter{case}{0}
\begin{case}
\textnormal{$D$ is a (i)- or (iii)-extension of $\tilde{C}_{m}$.}
\end{case}

Note for this case that $D$ contains only one vertex of ${S_{m+1} \cup V(K_{m+1})}$ by definition of (i)- or (iii)-extension, respectively, and because $\tilde{C}_{m}$ contains no vertex of ${S_{m+1} \cup V(K_{m+1})}$ by construction.
All paths $P_{i}$ of $\tilde{C}_{m}$ stay the same for $D$ for all $i$ with ${1 \leq i \leq m}$.
To check that this is valid, suppose we lose an edge of some path $P_{i'}$ with ${1 \leq i' \leq m}$.
If $D$ is a (i)-extension, there is precisely one edge in ${E(\tilde{C}_{m}) \setminus E(D)}$.
The endvertices of this edge are endvertices of the two edges in ${E(D) \setminus E(\tilde{C}_{m})}$, which are both incident with $s^{m+1}_1$. Since each edge of $P_{i'}$ has at least one endvertex in $K_{i'}$, the cycle $D$ contains at least one edge which is incident with ${s^{m+1}_1 \in S_{m+1}}$ and a vertex $v$ in $K_{i'}$.
This is a contradiction because ${i' \neq m+1}$ and so $s^{m+1}_1$ and $v$ lie in different components of $G - S_{i'}$ by definition of $S_{i'}$ and $K_{i'}$ together with Lemma~\ref{struct_toll}.

If $D$ is a (iii)-extension, ${E(D) \setminus E(\tilde{C}_{m})}$ contains precisely three edges, of which two, say $f$ and $g$, are incident with $s^{m+1}_1$ by definition of (iii)-extension with target $s^{m+1}_1$.
Let $e$ be the edge in ${E(D) \setminus E(\tilde{C}_{m})}$ which is different from $f$ and $g$.
The set ${E(\tilde{C}_{m}) \setminus E(D)}$ contains precisely two edges, of which each has a common endvertex with $e$ and with either $f$ or $g$.
So at least one of the edges in ${E(\tilde{C}_{m}) \setminus E(D)}$ lies on $P_{i'}$ and has therefore an endvertex in $K_{i'}$.
With the same argument as for $D$ being a (i)-extension, we know that $f$ and $g$ cannot have endvertices in $K_{i'}$.
So $e$ has one endvertex $y^+$ in $K_{i'}$ and since all vertices of ${V(\tilde{C}_{m}) \cap (S_{i'} \cup V(K_{i'}))}$ lie on the path $P_{i'}$ whose endvertices lie in $S_{i'}$, we get that $y$ is an endvertex of either $f$ or $g$ which lies in $S_{i'}$ by definition of (iii)-extension.
So the other endvertex of $e$ lies in another component of $G - S_{i'}$ than $K_{i'}$.
This contradicts the fact that $S_{i'}$ is a separator.

Now we pick an extension $D_1$ of $D$ with target $a \in N(s^{m+1}_1) \cap V(K_{m+1})$ (see Figure~4).
Since the cycle $D$ contains only one vertex of ${S_{m+1} \cup V(K_{m+1})}$ and $S_{m+1}$ is a separator, we know that $a$ has only one neighbour on $D$.
Hence, $D_1$ must be a (ii)-extension of $D$ and the neighbour $s^{m+1}_2$ of $a$ in $D_1$ which does not lie in $D$ must be an element of $S_{m+1}$.
Otherwise, $S_{m+1}$ would not separate ${a \in V(K_{m+1})}$ from $V(K_0)$, which is a contradiction.
Now set
$$\tilde{C}_{m+1} = D_1 \; \textnormal{  and  } \; P_{m+1} = s^{m+1}_1as^{m+1}_2.$$
For every $i$ with $1 \leq i \leq m$, we take the same paths as for $\tilde{C}_m$.
This setting fulfills the required conditions because
$$V(\tilde{C}_{m+1}) \setminus V(D) \subseteq S_{m+1} \cup V(K_{m+1}) \; \textnormal{  and  } \; E(D) \setminus E(\tilde{C}_{m+1}) = \lbrace e \rbrace$$
where the edge $e$ does not lie on any path $P_i$ for $1 \leq i \leq m$ since it is incident \linebreak with $s^{m+1}_1$.
This completes the observation of Case~1.

\begin{figure}[htbp]
\centering
\includegraphics[width=10cm]{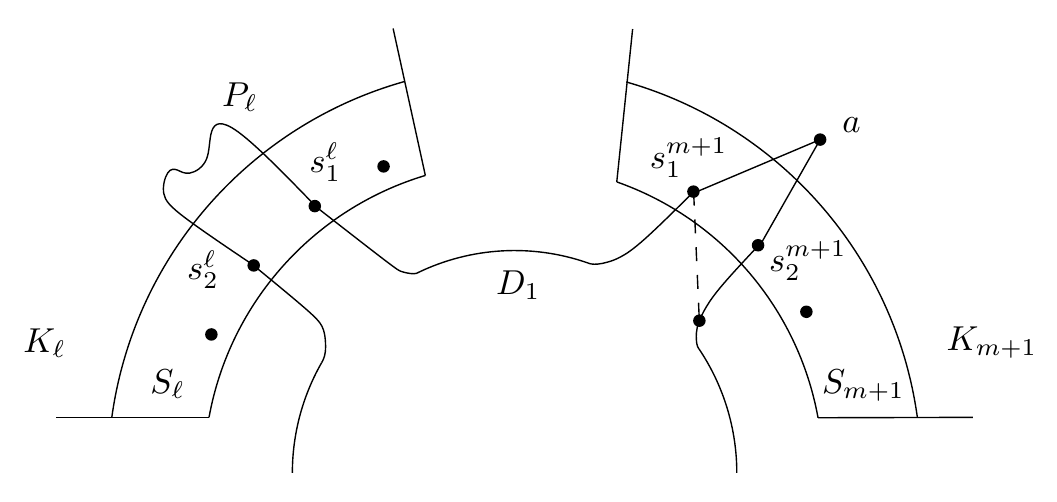}
\caption{Situation in Case~1.}
\label{lemma_4_3_case_1}
\end{figure}

\begin{case}
\textnormal{$D$ is a (ii)-extension of $\tilde{C}_{m}$.}
\end{case}

Let $x$ be the neighbour of $s^{m+1}_1$ in $D$ which does not lie in $V(\tilde{C}_{m})$ and let $ys^{m+1}_1xz$ be the path which replaces the edge $yz$ in $\tilde{C}_{m}$ to form $D$.
We know by construction that
$$V(K_0) \subseteq V(\tilde{C}_{m}) \; \textnormal{  and  } \; V(\tilde{C}_{m}) \cap S_{m+1} = \emptyset$$
hold.
Furthermore, ${x \in N(s^{m+1}_1) \cap N(\tilde{C}_{m})}$ is valid by definition of (ii)-extension.
Hence, $x$ is an element of $S_{\ell}$ for some $\ell$ with ${1 \leq \ell \leq k}$.
Here we distinguish three subcases:

\begin{subcase}
\textnormal{The equation $\ell = m+1$ holds.}
\end{subcase}

In this situation, we pick two vertices $a, b$ such that
$$a \in N(s^{m+1}_1) \cap V(K_{m+1})  \; \textnormal{  and  } \;  {b \in N(x) \cap V(K_{m+1})}.$$
Now we obtain a cycle $D'$ by replacing the edge $s^{m+1}_1x$ of $D$ by the path
$$P_{m+1} = s^{m+1}_1aT_{m+1}bx.$$
In this case, $x$ plays the role of $s^{m+1}_{2}$.
By setting $\tilde{C}_{m+1} = D'$ and taking $P_{m+1}$ together with all paths $P_i$ of $\tilde{C}_{m}$ for every $i$ with ${1 \leq i \leq m}$, all required properties are fulfilled.

\begin{subcase}
\textnormal{The relations $\ell \leq m$ and $z \in S_{\ell}$ are valid.}
\end{subcase}

In this case, we know that $y$ is no element of $S_{\ell}$ because $\tilde{C}_{m}$ contains only one vertex ${s^{\ell} \in S_{\ell}}$ that is different from $z$ by construction of $\tilde{C}_{m}$, but $s^{\ell}$ is no neighbour of $z$ in $\tilde{C}_{m}$ as $y$ is because of the $s^{\ell}$--$z$ path $P_{\ell}$ of $\tilde{C}_{m}$, which consists not just of one edge (see Figure~5).
Furthermore, $y$ is no element of $K_{\ell}$ because otherwise the edge $ys^{m+1}_1$ would show that $S_{\ell}$ does not separate $K_{\ell}$ from the component of $G - S_{\ell}$ which contains $s^{m+1}_1$.
This contradicts the definition of $S_{\ell}$.
By construction of $\tilde{C}_{m}$, we know that $\tilde{C}_{m}$ and, therefore, also $D$ contain only one $s^{\ell}$--$z$ path, namely $P_{\ell}$, whose interior vertices lie in $K_{\ell}$.
Now we move on with a slightly different cycle. In $D$, we replace the path ${s^{\ell}P_{\ell}zx}$ by the path ${s^{\ell}a'T_{\ell}b'x}$ where
$$a' \in N(s^{\ell}) \cap V(K_{\ell})  \; \textnormal{  and  } \;  b' \in N(x) \cap V(K_{\ell})$$
to obtain a new cycle $D''$.
Since $D''$ contains only one vertex of $S_{m+1}$ and still precisely two vertices of $S_i$ that are joined by a path through $K_i$ for every $i \leq m$, we can proceed as we did with $D$ in Case~1 and get the desired cycle $\tilde{C}_{m+1}$.

\begin{figure}[htbp]
\centering
\includegraphics[width=9.5cm]{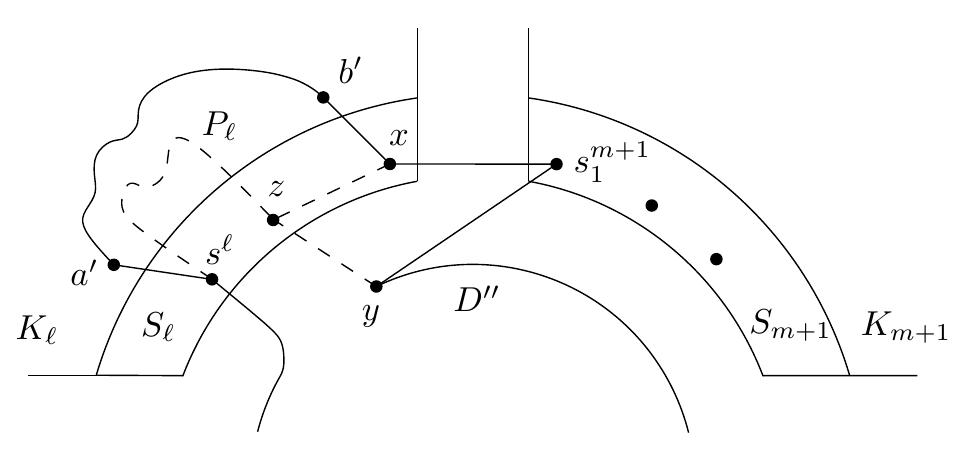}
\caption{Situation in Subcase~2.2.}
\label{lemma_4_3_case_2_2}
\end{figure}

\begin{subcase}
\textnormal{The relations $\ell \neq m+1$ and $z \notin S_{\ell}$ are true.}
\end{subcase}

If $z$ lies not in $K_{\ell}$, the vertices $s^{m+1}_1$ and $z$ are in the same component of $G-S_{\ell}$ and are neighbours of $x \in S_{\ell}$.
So by Lemma~\ref{complete}, the vertices $s^{m+1}_1$ and $z$ are adjacent.
Now we use the (i)-extension of $\tilde{C}_{m}$ which is formed by replacing the edge $yz$ of $\tilde{C}_{m}$ by the path ${ys^{m+1}_1z}$ instead of $D$ and proceed as in Case~1.
If $z$ is a vertex of $K_{\ell}$, we get that $y$ lies in $S_{\ell}$ because $y$ and $z$ are consecutive in $\tilde{C}_{m}$ and $y$ cannot be an element of $K_{\ell}$ since otherwise, the edge $ys^{m+1}_1$ would connect $K_{\ell}$ with the component of $G - S_{\ell}$ which contains $s^{m+1}_1$ and contradict the definition of $S_{\ell}$ (see Figure~6).
By construction, we know that the neighbour $w$ of $y$ in $\tilde{C}_{m}$ which is different from $z$ is no vertex of ${V(K_{\ell}) \cup S_{\ell}}$.
Hence, $s^{m+1}_1$ and $w$ lie in the same component of $G-S_{\ell}$ and are neighbours of $y$.
As before, Lemma~\ref{complete} implies that $s^{m+1}_1$ and $w$ are adjacent.
So we can use the (i)-extension of $\tilde{C}_{m}$ which arises by replacing the edge $wy$ of $\tilde{C}_{m}$ by the path $ws^{m+1}_1y$ instead of $D$ and proceed again as in Case~1.
Now the description of how to deal with Case~2 is complete too and we get the desired sequence $(\tilde{C}_{i})$ of cycles with $0 \leq i \leq k$.

\begin{figure}[htbp]
\centering
\includegraphics[width=9.5cm]{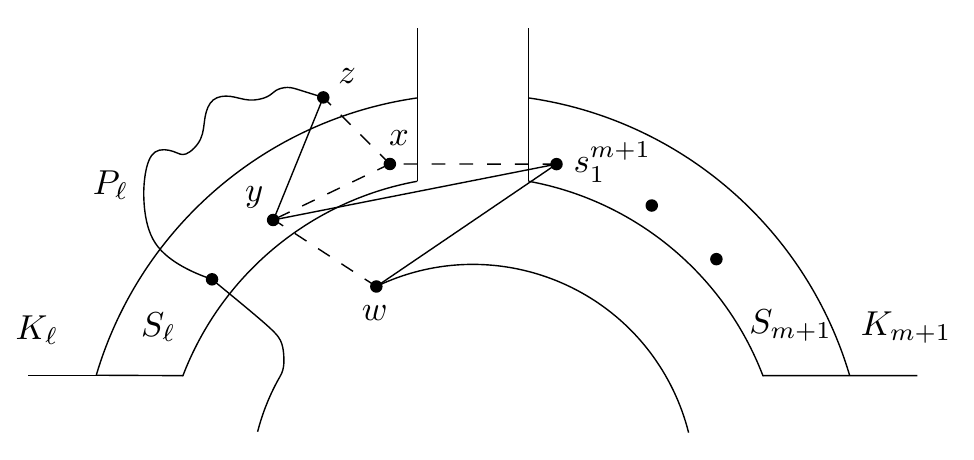}
\caption{Situation in Subcase~2.3 if $z$ lies in $K_{\ell}$.}
\label{lemma_4_3_case_2_3}
\end{figure}

For the next step, let $(\hat{C}_{i})$ be an extension sequence of $\tilde{C}_{k}$ where the targets are always chosen from ${\bigcup^k_{p=1} T_{p}}$ until a cycle of this sequence contains all vertices of ${\bigcup^k_{p=1} T_{p}}$.
The targets can always be chosen from $\bigcup^k_{p=1} T_{p}$ because each $T_p$ is connected and $\tilde{C}_{k}$ contains at least one vertex of each $T_p$ by construction.
Furthermore, we build only finitely many extensions since each tree $T_p$ is finite.
Now we prove the following claim:

\setcounter{claim}{0}
\begin{claim}
Each cycle $\hat{C}_{i}$ of the extension sequence hits the cut ${\delta(S_{j} \cup V(K_{j}))}$ precisely twice for each $j$ with ${1 \leq j \leq k}$.
\end{claim}

We prove this statement inductively.
We know that ${\hat{C}_{0} = \tilde{C}_{k}}$ fulfills the condition by its construction.
So assume that $\hat{C}_{n}$ fulfills the condition for some $n$ with ${0 \leq n \leq k-1}$ and consider $\hat{C}_{n+1}$.
Let ${t \in T_{j'}}$ be the target of the extension $\hat{C}_{n+1}$ for some $j'$ with ${1 \leq j' \leq k}$.
Since $\hat{C}_{n+1}$ is a cycle that has vertices in ${S_{j'} \cup V(K_{j'})}$ and ${V \setminus (S_{j'} \cup V(K_{j'}))}$, the cycle $\hat{C}_{n+1}$ must hit the cut ${\delta(S_{j'} \cup V(K_{j'}))}$ at least twice and in an even number of edges.
Furthermore, we know that both edges of $\hat{C}_{n+1}$ which are incident with $t$ are no elements of the cut ${\delta(S_{j'} \cup V(K_{j'}))}$ because all neighbours of $t$ lie in ${V(K_{j'}) \cup S_{j'}}$.
Using the induction hypothesis, we get that $\hat{C}_{n}$ hits the cut ${\delta(S_{j'} \cup V(K_{j'}))}$ precisely twice.
Then, by the definitions of the three types of extensions, $\hat{C}_{n+1}$ cannot meet the cut ${\delta(S_{j'} \cup V(K_{j'}))}$ four times, which implies that it hits the cut exactly twice.
This completes the induction and the proof of the claim.
\newline

Now we construct a sequence $(\hat{D}_i)$ of cycles where the vertex set of the last cycle of this sequence contains all elements of $\mathscr{S}$ and the following properties hold for each $i \geq 0$ if the corresponding cycles are defined:
\vspace{2pt}

\begin{itemize}
\setlength\itemsep{3pt}
\item ${V(K_0) \cup N_3(\mathscr{S}) \subseteq V(\hat{D}_i)}$.
\item ${V(\hat{D}_{i+1}) \setminus V(\hat{D}_{i}) \subseteq \mathscr{S}}$.
\item $1 \leq |V(\hat{D}_{i+1}) \setminus V(\hat{D}_{i})| \leq 2$.
\end{itemize}

\vspace{2pt}
\noindent Furthermore, for every $i \geq 0$ such that $\hat{D}_{i}$ is defined there shall exist vertex sets $M^i_j$ for ${1 \leq j \leq k}$ such that the following properties are fulfilled:
\vspace{2pt}

\begin{itemize}
\setlength\itemsep{3pt}
\item $V(K_j) \setminus N(S_j) \subseteq M^i_j \subseteq V(K_j) \cup \mathscr{S} \cup N(\mathscr{S})$.
\item $|E(\hat{D}_{i}) \cap \delta(M^i_j)| = 2$.
\end{itemize}

\vspace{2pt}
We begin by setting ${\hat{D}_{0} = \hat{C}_{k}}$.
We know that ${V(K_0) \cup N_3(\mathscr{S}) \subseteq V(\hat{D}_{0})}$ holds by construction.
Additionally, Claim~1 implies that ${M^{0}_{j} = S_{j} \cup V(K_{j})}$ is a valid choice for every $j$ with ${1 \leq j \leq k}$.

Now assume we have already constructed the sequence up to $\hat{D}_{m}$ and there is still a vertex ${u \in \mathscr{S} \setminus V(\hat{D}_{m})}$, say ${u \in S_{i'}}$ for some $i'$ with ${1 \leq i' \leq k}$.
If $u$ has two neighbours $w_1, w_2$ which are adjacent in $\hat{D}_{m}$, we define $\hat{D}_{m+1}$ as the (i)-extension of $\hat{D}_{m}$ where the edge $w_1w_2$ is replaced by the path $w_1uw_2$.
We define the sets $M^{m+1}_{j}$ as follows for every $j$ with ${1 \leq j \leq k}$:
\[M^{m+1}_{j} = 
\begin{cases}
M^{m}_{j} \cup \lbrace u \rbrace &\mbox{if } w_1 \in M^{m}_{j}  \textnormal{ or }  w_2 \in M^{m}_{j} \\
M^{m}_{j} \setminus \lbrace u \rbrace & \mbox{otherwise}.
\end{cases} \]
All required conditions are fulfilled by this definition.

So let us assume that $u$ does not have two neighbours which are consecutive in $\hat{D}_{m}$.
Since $G$ is claw-free, we know that for every vertex ${w \in N(u) \cap V(\hat{D}_{m})}$, the vertices $w^+$ and $w^-$ are adjacent in $G$.
Now let ${w_1 \in N(u) \cap V(\hat{D}_{m})}$ be fixed and consider the cycle $\hat{D}^u_{m}$ which is formed by replacing the path $w^+ww^-$ in $\hat{D}_{m}$ by the edge $w^+w^-$ for every ${w \in (N(u) \cap V(\hat{D}_{m})) \setminus \lbrace w_1 \rbrace}$.
By Lemma~\ref{Asra-enlarge}, there exists an extension of $\hat{D}^u_{m}$ with target $u$, but since $w_1$ is the only neighbour of $u$ on $\hat{D}^u_{m}$, all extensions of $\hat{D}^u_{m}$ with target $u$ must be (ii)-extensions.
So let there be a (ii)-extension of $\hat{D}^u_{m}$ with target $u$ where the edge ${w_1w_2 \in E(\hat{D}^u_{m})}$ is replaced by the path $w_1uhw_2$.
If ${h \notin V(\hat{D}_{m})}$ holds, then there is also a (ii)-extension of $\hat{D}_{m}$ which we set as $\hat{D}_{m+1}$.
The sets $M^{m+1}_{j}$ are defined for every $j$ with ${1 \leq j \leq k}$ in the following way:
\[M^{m+1}_{j} = 
\begin{cases}
M^{m}_{j} \cup \lbrace u, h \rbrace &\mbox{if } w_1 \in M^{m}_{j}  \textnormal{ or }  w_2 \in M^{m}_{j} \\
M^{m}_{j} \setminus \lbrace u, h \rbrace & \mbox{otherwise}.
\end{cases} \]
Using this definition, all required conditions are again fulfilled.

It remains to handle the case when ${h \in V(\hat{D}_{m})}$ is true.
In this situation, we build $\hat{D}_{m+1}$ as follows.
We take $\hat{D}_{m}$, replace the path $h^+hh^-$ by the edge $h^+h^-$ and the edge $w_1w_2$ by the path $w_1uhw_2$ (see Figure~7).
Furthermore, we set $M^{m+1}_{j}$ as before in the case with ${h \notin V(\hat{D}_{m})}$ for every $j$ with ${1 \leq j \leq k}$.

We make some remarks to see that
$$|E(\hat{D}_{m+1}) \cap \delta(M^{m+1}_j)| = 2$$
holds for every $j$ with $1 \leq j \leq k$.
Note at first that neither
$$V(\hat{D}_{m+1}) \cap M^{m+1}_{j} = \emptyset  \; \textnormal{  nor  } \;  V(\hat{D}_{m+1}) \setminus M^{m+1}_{j} = \emptyset$$
because ${V(\hat{D}_{m}) \cap M^{m}_{j}}$ as well as ${V(\hat{D}_{m}) \setminus M^{m}_{j}}$ contains vertices with distance at least $2$ to $\mathscr{S}$.
These facts are due to the relation
$$V(T_j) \subseteq V(\hat{D}_{m})$$
and our assumption on $C$ that
$$V(C) \setminus N(N(C)) \neq \emptyset$$
combined with the relation
$$V(C) \subseteq V(\hat{D}_{m}).$$
Next, note that if the path $w_1uhw_2$ meets some cut $\delta(M^{m+1}_{j})$, it can meet it only once by definition of $M^{m+1}_{j}$ and then the edge $w_1w_2$ must be an element of $\delta(M^{m}_{j})$.
Similarly, if ${h^+h^- \in \delta(M^{m+1}_{j})}$ holds, the path $h^+hh^-$ meets the cut $\delta(M^{m}_{j})$ only once.
Both edges of this path cannot lie in the cut $\delta(M^{m}_{j})$ because
$$|E(\hat{D}_{m}) \cap \delta(M^{m}_{j})| = 2$$
holds and so $h$ would be the only vertex in
$$V(\hat{D}_{m}) \cap M^{m}_{j}   \; \textnormal{  or  } \;   V(\hat{D}_{m}) \setminus M^{m}_{j}.$$
This would be a contradiction since both of these sets contain vertices with distance at least $2$ to $\mathscr{S}$.
We know that $\hat{D}_{m+1}$ meets the cut $\delta(M^{m+1}_{j})$ at least twice because $\hat{D}_{m+1}$ is a cycle that has vertices in $M^{m+1}_{j}$ and ${V \setminus M^{m+1}_{j}}$.
Since
$$|E(\hat{D}_{m}) \cap \delta(M^{m}_j)| = 2$$
holds by construction and $\hat{D}_{m+1}$ is formed from $\hat{D}_{m}$ be deleting the edges $w_1w_2$, $h^+h$ and $hh^-$ but adding $w_1u$, $uh$, $hw_2$ and $h^+h^-$, the equation
$$|E(\hat{D}_{m+1}) \cap \delta(M^{m+1}_j)| = 2$$
is valid.
So the cycle $\hat{D}_{m+1}$ and the sets $M^{m+1}_j$ fulfill all required conditions.

Since $G$ is locally finite and $\mathscr{S}$ is a subset of $N(C)$, we know that $\mathscr{S}$ is finite.
Then there must exist an integer $M$ such that $\hat{D}_{M}$ contains all vertices of ${V(K_0) \cup \mathscr{S} \cup N_3(\mathscr{S})}$.
We set
$$C' = \hat{D}_{M}  \; \textnormal{  and  } \;   M_j = M^{M}_j$$
for every $j$ with $1 \leq j \leq k$.
The cycle $C'$ and the sets $M_j$ show that statement~(i) and (ii) of the lemma are true.

\begin{figure}[htbp]
\centering
\includegraphics[width=10.5cm]{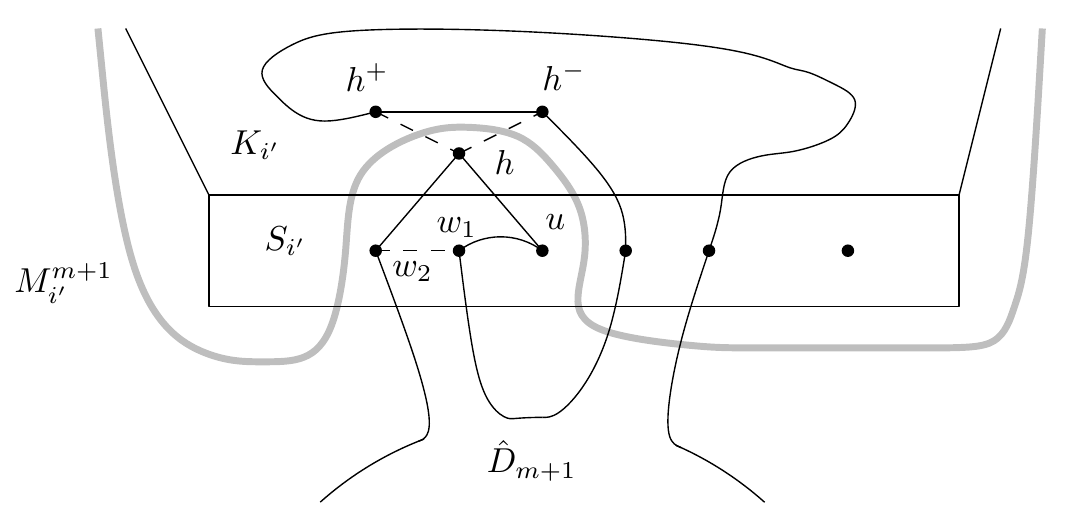}
\caption{The cycle $\hat{D}_{m+1}$ for the case $h \in V(\hat{D}_m)$.}
\label{lemma_4_3_M_j}
\end{figure}

It remains to check that statement~(iii) of the lemma holds for the cycle $C'$.
Note for the inclusion
$${E(C - N(N(C))) \subseteq E(C')}$$
that if we have lost edges of the cycle $C$, then they have at least one endvertex in $N(N(C))$ because of the definition of extension and the operation where we replaced a path $h^+hh^-$ by the edge $h^+h^-$ and $h$ is a neighbour of some vertex in $\mathscr{S}$.
This shows the first part of statement~(iii).

Note for the other part of statement~(iii) that we got edges in ${E(C') \setminus E(C)}$ only by building extensions with targets in ${V \setminus V(C)}$, by taking paths whose vertices lie entirely in ${V \setminus V(C)}$ as in Case~2 and by replacing paths $h^+hh^-$ by the edge $h^+h^-$ where $h$ lies on some cycle and is a neighbour of some vertex in ${\mathscr{S} \subseteq N(C)}$.
Since $h^+$ and $h^-$ are neighbours of $h$, we get that
$$\lbrace h^+, h^- \rbrace \subseteq N_2(N(C)) \cup N(C).$$
Next let us check the location of the relevant edges of extensions.
Let $Z'$ be an extension of some cycle $Z$ in $G$ with target in ${V \setminus V(C)}$.
Then we know for each edge ${e = uv \in E(Z') \setminus E(Z)}$ that
$$\lbrace u, v \rbrace \subseteq (V \setminus V(C)) \cup N_2(N(C))$$
holds by the definition of extension.
Putting these observations together, statement~(iii) is completely proved.
So the proof of the whole lemma is done.
\end{proof}

Having Lemma~\ref{Asra-cut-1} in our toolkit, we now prove Theorem~\ref{Asra-loc-fin}.
As mentioned before, the rough idea of the proof is to construct a sequence of cycles together with certain vertex sets such that we get a Hamilton circle as a limit object from the sequence of cycles using Lemma~\ref{HC-extract}.
For the construction of these objects, we use Lemma~\ref{Asra-cut-1}. 

\begin{proof}[Proof of Theorem~\ref{Asra-loc-fin}]
Let $G = (V, E)$ be a locally finite, connected, claw-free graph which satisfies $(\ast)$ and has at least three vertices.
We may assume that $G$ is infinite because for finite $G$ the statement follows from Theorem~\ref{Asra-fin}.
First we define a sequence $(C_i)_{i \in \mathbb{N}}$ of cycles of $G$ such that
$$V(C_i) \setminus N(N(C_i)) \neq \emptyset$$
holds for every ${i \in \mathbb{N}}$.
Additionally, we define an integer sequence ${(k_i)_{i \in \mathbb{N} \setminus \lbrace 0 \rbrace}}$ and vertex sets ${M^i_j \subseteq V(G)}$ where ${i \in \mathbb{N} \setminus \lbrace 0 \rbrace}$ and for every such $i$ the inequality chain ${1 \leq j \leq k_i}$ is satisfied.

We start by taking an arbitrary cycle as $\tilde{C}$.
Note that $G$ contains a cycle since $G$ is connected, has three vertices and satisfies $(\ast)$.
The argumentation is the the same as in the beginning of the proof of Theorem~\ref{Asra-fin}.
Now take an extension sequence of $\tilde{C}$ where we choose the targets of the extensions always from $N(\tilde{C})$.
Since $G$ is locally finite, $N(\tilde{C})$ is finite and the extension sequence ends after finitely many steps.
We set $C_0$ as the last cycle of such an extension sequence of $\tilde{C}$.
Note that
$$V(\tilde{C}) \subseteq V(C_0) \setminus N(N(C_0)).$$

Now suppose we have already defined the sequence of cycles up to length $m+1$ for some ${m \geq 0}$ together with the integer sequence up to $k_m$ and the vertex sets $M^i_j$ for every ${i \leq m}$ where $j$ satisfies always ${1 \leq j \leq k_i}$.
Then let
$$\mathscr{S}^{m+1} \subseteq N(C_{m})$$
be a finite minimal vertex set such that every ray which starts in $V(C_m)$ has to meet $\mathscr{S}^{m+1}$.
Such a set exists because $G$ is locally finite, which implies that $N(C_{m})$ is finite.
Hence, we get $\mathscr{S}^{m+1}$ by sorting out vertices from $N(C_m)$.
Next we set $k_{m+1}$ as the integer we get from Lemma~\ref{struct_toll}.
Furthermore, let
$$S^{m+1}_1, \ldots, S^{m+1}_{k_{m+1}}$$
be the minimal separators and
$$K^{m+1}_0, \ldots, K^{m+1}_{k_{m+1}}$$
be the components of ${G-\mathscr{S}^{m+1}}$ we get from Lemma~\ref{struct_toll}.
Applying Lemma~\ref{Asra-cut-1} with these objects and the cycle $C_m$, we obtain a cycle which we set as $C_{m+1}$.
We also get vertex sets for every $j$ with ${1 \leq j \leq k_{m+1}}$ which we choose for the sets $M^{m+1}_j$ (see Figure~8).

We want to use Lemma~\ref{HC-extract} to prove that $G$ is Hamiltonian.
The next claim ensures that all required conditions are fulfilled to apply Lemma~\ref{HC-extract}.

\setcounter{claim}{0}
\begin{claim}
\textnormal{
\begin{enumerate}[\normalfont(a)]
\item \textit{For every vertex $v$ of $G$, there exists an integer ${j \geq 0}$ such that ${v \in V(C_i)}$ holds for every ${i \geq j}$.}
\item \textit{For every ${i \geq 1}$ and $j$ with ${1 \leq j \leq k_i}$, the cut $\delta(M^i_j)$ is finite.}
\item \textit{For every end $\omega$ of $G$, there is a function ${f : \mathbb{N} \setminus \lbrace 0 \rbrace \longrightarrow \mathbb{N}}$ such that the inclusion ${{M^{j}_{f(j)} \subseteq M^i_{f(i)}}}$ holds for all integers $i, j$ with ${1 \leq i \leq j}$ and the equation ${M_{\omega}:= \bigcap^{\infty}_{i=1} \overline{M^i_{f(i)}} = \lbrace \omega \rbrace}$ is true.}
\item \textit{${E(C_i) \cap E(C_j) \subseteq E(C_{j+1})}$ holds for all integers $i$ and $j$ with ${0 \leq i < j}$.}
\item \textit{The equations ${E(C_i) \cap \delta(M^p_j) = E(C_p) \cap \delta(M^p_j)}$ and ${|E(C_i) \cap \delta(M^p_j)| = 2}$ hold for each triple $(i, p, j)$ which satisfies ${1 \leq p \leq i}$ and ${1 \leq j \leq k_p}$.}
\end{enumerate}
}
\end{claim}

\begin{figure}[htbp]
\centering
\includegraphics[width=8.5cm]{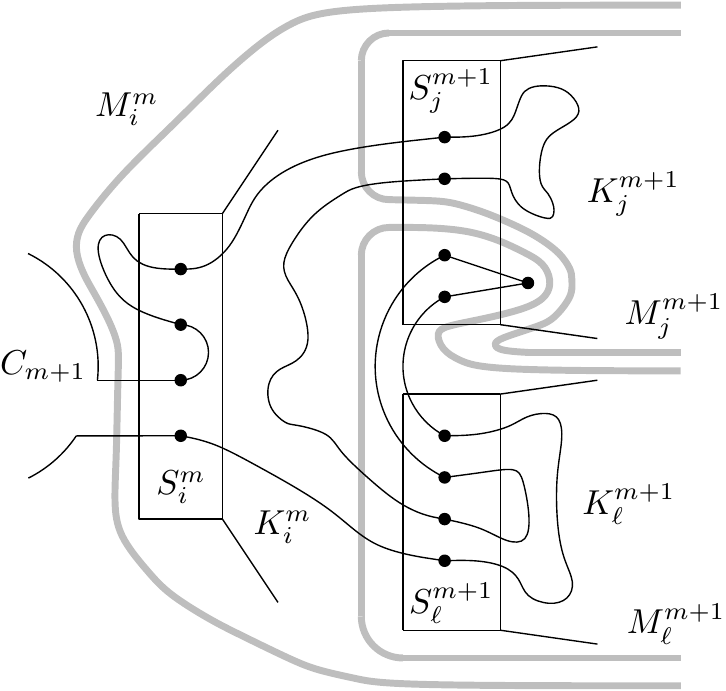}
\caption{The cycle $C_{m+1}$ and the vertex sets we got from Lemma~\ref{Asra-cut-1}.}
\label{thm_1_4_C_m+1}
\end{figure}

Note that the inclusions
$$V(K^{i}_0) \cup \mathscr{S}^{i} \cup N_3(\mathscr{S}^{i}) \subseteq V(C_{i}) \subseteq V(K^{i+1}_0)$$
and the equation ${N(K^{i}_0) = \mathscr{S}^{i}}$ hold for every ${i \geq 1}$ by definition of the cycles together with Lemma~\ref{Asra-cut-1}~(i) and by Lemma~\ref{struct_toll}.
Since $G$ is connected, statement~(a) follows.

We fix an arbitrary integer ${i \geq 1}$ and some $j$ with ${1 \leq j \leq k_i}$ for the proof of statement~(b).
By definition of the set $M^i_j$ and Lemma~\ref{Asra-cut-1}~(ii), the inclusions
$${V(K^i_j) \setminus N(S^i_j) \subseteq M^i_j \subseteq V(K^i_j) \cup \mathscr{S}^i \cup N(\mathscr{S}^i)}$$
are true.
Now the definitions of $K^i_j$ and $\mathscr{S}^i$ imply that $N(M^i_j)$ is a subset of ${V(K^i_0) \cup \mathscr{S}^i \cup N_2(\mathscr{S}^i)}$.
Using that $V(K^i_0)$ and $\mathscr{S}^i$ are finite sets by definition and that $G$ is locally finite, we obtain that $\delta(M^i_j)$ is a finite cut.

We fix an arbitrary end $\omega$ of $G$ for statement~(c).
Now we use that for every $i \geq 1$ the end $\omega$ is contained in precisely one of the closures ${\overline{K^i_{1}}, \ldots, \overline{K^i_{k_i}}}$, say ${\omega \in \overline{K^i_{j}}}$ where ${1 \leq j \leq k_i}$.
Then set ${f(i) = j}$.
First we prove that
$${M^{j}_{f(j)} \subseteq M^i_{f(i)}}$$
holds for all integers $i, j$ with ${1 \leq i \leq j}$.
For this, it suffices to show that the inclusion
$${M^{i+1}_{f(i+1)} \subseteq M^i_{f(i)}}$$
is true for every ${i \geq 1}$.
We get that the inclusions
$${V(K^{i}_{f(i)}) \setminus N(S^{i}_{f(i)}) \subseteq M^{i}_{f(i)} \subseteq V(K^{i}_{f(i)}) \cup \mathscr{S}^{i} \cup N(\mathscr{S}^{i})}$$
hold for every ${i \geq 1}$ by definition of the set $M^{i}_{f(i)}$ and Lemma~\ref{Asra-cut-1}~(ii).
Note that ${\mathscr{S}^{i+1} \cup N(\mathscr{S}^{i+1})}$ is not necessarily a subset of $V(K^{i}_{f(i)})$.
Because of this, we have to look a bit more carefully at the set $M^{i}_{f(i)}$.
For our purpose, it suffices to prove that $M^{i+1}_{f(i+1)}$ is a subset of ${V(K^{i}_{f(i)}) \setminus N(S^{i}_{f(i)})}$ for every ${i \geq 1}$.
Suppose this is not true.
Then, using the definition of $K^{n}_{f(n)}$ and $S^{n+1}_{\ell}$ together with Lemma~\ref{struct_toll}, there exists an integer ${n \geq 1}$ and an integer ${\ell \neq f(n+1)}$ with ${1 \leq \ell \leq k_{n+1}}$ such that
$${V(K^{n}_{f(n)}) \cap (S^{n+1}_{\ell} \cup N(S^{n+1}_{\ell})) = \emptyset}$$
and
$${X = M^{n+1}_{f(n+1)} \cap (S^{n+1}_{\ell} \cup N(S^{n+1}_{\ell}))} \neq \emptyset.$$
Since the inclusion
$${V(K^{n+1}_0) \cup \mathscr{S}^{n+1} \cup N_3(\mathscr{S}^{n+1}) \subseteq V(C_{n+1})}$$
holds by definition of $C_{n+1}$ and Lemma~\ref{Asra-cut-1}~(i), we get that $C_{n+1}$ has vertices in $X$ and ${V \setminus X}$.
So $C_{n+1}$ hits the cut $\delta(X)$ at least twice.
Furthermore, we know that no edge of $\delta(X)$ is an edge of $K^{n}_{f(n)}$.
The inclusion
$${M^{n+1}_{f(n+1)} \subseteq V(K^{n+1}_{f(n+1)}) \cup \mathscr{S}^{n+1} \cup N(\mathscr{S}^{n+1})}$$
and the definition of $X$ ensure that
$$\delta(X) \subseteq \delta(M^{n+1}_{f(n+1)}).$$
Hence, $E(C_{n+1})$ contains no edges of ${\delta(M^{n+1}_{f(n+1)}) \setminus \delta(X)}$ because of Lemma~\ref{Asra-cut-1}~(ii).
This yields a contradiction.
In order to show that $E(C_{n+1})$ must use at least one edge of ${\delta(M^{n+1}_{f(n+1)}) \setminus \delta(X)}$, we define the set
$${Y = M^{n+1}_{f(n+1)} \cap V(K^{n}_{f(n)})}.$$
Using the inclusions
$${V(K^{n+1}_{f(n+1)}) \setminus N(S^{n+1}_{f(n+1)}) \subseteq M^{n+1}_{f(n+1)} \subseteq V(K^{n+1}_{f(n+1)}) \cup \mathscr{S}^{n+1} \cup N(\mathscr{S}^{n+1})}$$
and the definitions of $K^{n}_{f(n)}$ and $K^{n+1}_{f(n+1)}$ together with Lemma~\ref{struct_toll}, it is ensured that $Y$ is not empty since the inclusion
$$V(K^{n+1}_{f(n+1)}) \subseteq V(K^{n}_{f(n)})$$
holds and that the inclusion
$${\delta(Y) \subseteq \delta(M^{n+1}_{f(n+1)}) \cap E(K^{n}_{f(n)})}$$
is true.
So the edge sets $\delta(X)$ and $\delta(Y)$ are disjoint.
Using the inclusion
$${V(K^{n+1}_0) \cup \mathscr{S}^{n+1} \cup N_3(\mathscr{S}^{n+1}) \subseteq V(C_{n+1})}$$
again, we obtain that $C_{n+1}$ contains vertices in $Y$ and ${V \setminus Y}$, which implies that $E(C_{n+1})$ contains at least two edges of $\delta(Y)$.
Since
$$\delta(Y) \subseteq \delta(M^{n+1}_{f(n+1)}) \setminus \delta(X),$$
we have the desired contradiction.
So the inclusion ${M^{j}_{f(j)} \subseteq M^i_{f(i)}}$ holds for all integers $i, j$ with ${1 \leq i \leq j}$.

It remains to show that ${M_{\omega} = \lbrace \omega \rbrace}$ is true.
As noted above, the inclusions
$${V(K^i_{f(i)}) \setminus N(S^i_{f(i)}) \subseteq M^i_{f(i)} \subseteq V(K^i_{f(i)}) \cup \mathscr{S}^i \cup N(\mathscr{S}^i)}$$
are true for every ${i \geq 1}$.
So $\omega$ is an element of $M_{\omega}$ by definition of the function $f$.
To show that $M_{\omega}$ contains no vertex of $G$ and no other end of $G$, fix some vertex $v \in V$ and some end ${\omega' \neq \omega}$ of $G$.
Now let $F$ be a finite set of vertices such that $\omega$ and $\omega'$ lie in closures of different components of ${G-F}$.
We take an integer ${p \geq 1}$ such that the following inclusion is fulfilled:
$$F \cup \lbrace v \rbrace \subseteq V(K^{p}_0).$$

To see that it is possible to find such an integer, note that each vertex ${w \in F \cup \lbrace v \rbrace}$ lies in some cycle $C_{\ell_w}$ where ${\ell_w \geq 0}$ by statement~(a).
The construction of the cycles and Lemma~\ref{Asra-cut-1}~(i) ensure that the inclusion
$$V(C_i) \subseteq V(C_{i+1})$$
holds for every ${i \geq 0}$.
Since ${F \cup \lbrace v \rbrace}$ is finite, we can set ${p-1}$ as the maximum of all integers $\ell_w$.
Now the definition of $K^{p}_0$ and Lemma~\ref{struct_toll} imply that $V(K^{p}_0)$ contains all vertices of ${F \cup \lbrace v \rbrace}$.

So $\omega$ and $\omega'$ are also in closures of different components of ${G-(V(K^{p}_0) \cup \mathscr{S}^p))}$.
As we have proved already, the set $M^{i+1}_{f(i+1)}$ is a subset of ${V(K^i_{f(i)}) \setminus N(S^i_{f(i)})}$ for every ${i \geq 1}$.
So $\omega'$ and $v$ do not lie in the set $\overline{M^{p+1}_{f(p+1)}}$, which implies that they cannot be elements of $M_{\omega}$.
Since each set $M^{i}_{f(i)}$ is a set of vertices, the intersection $M_{\omega}$ cannot contain inner points of edges.
Therefore, the equation ${M_{\omega} = \lbrace \omega \rbrace}$ is valid and statement~(c) is true.

To prove statement~(d), take an edge ${e \in E(C_i) \cap E(C_j)}$ for arbitrary integers $i$ and $j$ that satisfy ${0 \leq i < j}$.
By definition of the cycles together with Lemma~\ref{Asra-cut-1}~(iii) and (i), we get that both endvertices of $e$ lie in ${V(C_i) \subseteq V(K^{i+1}_0)}$ and that the inclusions
$${V(K^{i+1}_0) \cup \mathscr{S}^{i+1} \cup N_3(\mathscr{S}^{i+1}) \subseteq V(C_{i+1}) \subseteq V(C_j)}$$
hold.
Using the equation
$${N(K^{i+1}_0) = \mathscr{S}^{i+1}},$$
we obtain that
$${e \in E(C_j - N(N(C_j)))}$$
is true.
So $e$ is an element of $E(C_{j+1})$ by definition of the cycles and Lemma~\ref{Asra-cut-1}~(iii).
Therefore, statement~(d) holds.

Let us fix an arbitrary ${p \geq 1}$ and $j$ with ${1 \leq j \leq k_p}$ for statement~(e).
The equation
$$|E(C_p) \cap \delta(M^p_j)| = 2$$
is true by definition of the cycles and Lemma~\ref{Asra-cut-1}~(ii).
So proving the equation
$$E(C_p) \cap \delta(M^p_j) = E(C_i) \cap \delta(M^p_j)$$ for every $i \geq p$ suffices to show that statement~(e) holds.
First we verify the inclusion
$${E(C_p) \cap \delta(M^p_j) \subseteq E(C_i) \cap \delta(M^p_j)}.$$
We know that for every ${p \geq 1}$ where ${1 \leq j \leq k_p}$ the inclusions
$${V(K^p_{j}) \setminus N(S^p_j) \subseteq M^p_{j} \subseteq V(K^p_{j}) \cup \mathscr{S}^p \cup N(\mathscr{S}^p)}$$
are true by definition of the set $M^p_j$ and Lemma~\ref{Asra-cut-1}~(ii).
So
$$\lbrace u,v \rbrace \subseteq \mathscr{S}^p \cup N_2(\mathscr{S}^p)$$
holds for each edge ${uv \in \delta(M^p_j)}$.
By definition of the cycles and Lemma~\ref{Asra-cut-1}~(i), we know further that the inclusion
$${V(K^{p}_0) \cup \mathscr{S}^p \cup N_3(\mathscr{S}^p) \subseteq V(C_i)}$$
is true for every ${i \geq p \geq 1}$.
Hence, if $uv$ is an edge of ${E(C_i) \cap \delta(M^p_j)}$, it lies also in ${E(C_i - N(N(C_i)))}$, which implies ${uv \in E(C_{i+1})}$ by Lemma~\ref{Asra-cut-1}~(iii).
So we get inductively that
$$E(C_p) \cap \delta(M^p_j) \subseteq E(C_i) \cap \delta(M^p_j)$$
holds for every ${i \geq 1}$.

It remains to check the opposing inclusion.
We do this by showing that for every ${i \geq p}$ the cycle $C_i$ contains no edges of $\delta(M^p_j)$ but the two which are also edges of $C_p$.
For this, we use induction on $i$.
The definition of $C_p$ and Lemma~\ref{Asra-cut-1}~(ii) ensure that the statement holds for $i=p$.
Next we fix an arbitrary $i > p$ and edge ${e = uv \in \delta(M^p_j) \setminus E(C_p)}$.
Using the induction hypothesis, we get that $e$ is no edge of $C_{i-1}$.
Now suppose for a contradiction that $e$ is an edge of $C_i$.
Then the definition of $C_i$ together with Lemma~\ref{Asra-cut-1}~(iii) implies that the inclusion
$${\lbrace u, v \rbrace \subseteq (V \setminus V(C_{i-1})) \cup N_2(N(C_{i-1}))}$$
holds.
This leads towards a contradiction because we already know that the inclusion
$$\lbrace u,v \rbrace \subseteq \mathscr{S}^p \cup N_2(\mathscr{S}^p)$$
is valid.
Both inclusions cannot be true at the same time since $C_{i-1}$ contains all vertices of ${V(K^p_0) \cup \mathscr{S}^p \cup N_3(\mathscr{S}^{p})}$ by definition of the cycle and Lemma~\ref{Asra-cut-1}~(i).
This completes the induction and, therefore, shows the equation
$$E(C_p) \cap \delta(M^p_j) = E(C_i) \cap \delta(M^p_j)$$
for every ${i \geq p}$.
Hence, statement~(e) is true and the proof of the claim is complete.
\newline

As mentioned before, we can now apply Lemma~\ref{HC-extract} using the sequence of cycles $(C_i)_{i \in \mathbb{N}}$, the integer sequence $(k_i)_{i \in \mathbb{N} \setminus \lbrace 0 \rbrace}$ and the vertex sets ${M^i_j}$ for every ${i \in \mathbb{N} \setminus \lbrace 0 \rbrace}$ and $j$ with ${1 \leq j \leq k_i}$ to obtain that $G$ is Hamiltonian.
\end{proof}

Now let us discuss how the proof of Theorem~\ref{Asra-loc-fin} depends on the assumption of being claw-free.
Lemma~\ref{complete} and Lemma~\ref{struct_toll} do not have to be true if our graph contains claws.
Without the second of these lemmas, the structure of the graph is not so clear anymore, but of course we can still find for every end a sequence of separators and components which captures the end.
The harder problem is that without Lemma~\ref{complete} it is not clear how to control the growth of the sequence of cycles through the separators, which we need in order to apply Lemma~\ref{HC-extract}.
In the proof of Lemma~\ref{Asra-cut-1} we made heavy use of Lemma~\ref{complete}.
So in order to make progress towards a version of Theorem~\ref{Asra-loc-fin} which does not depend on the assumption of being claw-free, we need to find a way to control the growth of sequences of cycles given by extensions (maybe just along separators) only using property $(\ast)$.


\begin{thebibliography}{xx}

\bibitem{asra-good} A. S. Asratian and N. K. Khachatrian, \textit{Some localization theorems on Hamiltonian circuits}, J. Combin. Theory Ser. B \textbf{49} (1990) 287--294.

\bibitem{brewster-funk} R. C. Brewster and D. Funk, \textit{On the hamiltonicity of line graphs of locally finite, 6-edge-connected graphs}, J. Graph Theory \textbf{71} (2012) 182--191.

\bibitem{circle} H. Bruhn and M. Stein, \textit{On end degrees and infinite cycles in locally finite graphs}, Combinatorica \textbf{27} (2007) 269--291.

\bibitem{bruhn-HC} H. Bruhn and X. Yu, \textit{Hamilton cycles in planar locally finite graphs}, SIAM J. Discrete \linebreak Math. \textbf{22} (2008) 1381--1392.

\bibitem{diestel_buch} R. Diestel, \textit{Graph Theory}, fourth ed., Springer-Verlag, 2012.

\bibitem{diestel_arx} R. Diestel, \textit{Locally finite graphs with ends: a topological approach},  arXiv:0912.4213v3  (2012).

\bibitem{diestel_kuehn_1} R. Diestel and D. K\"{u}hn, \textit{On infinite cycles I}, Combinatorica \textbf{24} (2004) 68--89.

\bibitem{diestel_kuehn_2} R. Diestel and D. K\"{u}hn, \textit{On infinite cycles II}, Combinatorica \textbf{24} (2004) 91--116.

\bibitem{diestel_kuehn_TST} R. Diestel and D. K\"{u}hn, \textit{Topological paths, cycles and spanning trees in infinite graphs}, Europ. J. Combinatorics \textbf{25} (2004) 835--862.

\bibitem{freud-equi} R. Diestel and D. K\"{u}hn, \textit{Graph-theoretical versus topological ends of graphs}, J. Combin. Theory Ser. B \textbf{87} (2003) 197--206.

\bibitem{dirac} G. A. Dirac, \textit{Some theorems on abstract graphs},  Proc. London Math. Soc. (3) \textbf{2} (1952) 69--81.

\bibitem{agelos-HC} A. Georgakopoulos, \textit{Infinite Hamilton cycles in squares of locally finite graphs}, Adv. \linebreak Math. \textbf{220} (2009) 670--705.

\bibitem{Ha_Leh_Po} M. Hamann, F. Lehner and J. Pott, \textit{Extending cycles locally to Hamilton cycles}, Electronic. J.~Comb. \textbf{23} (2016) P1.49.

\bibitem{heuer_Inf_ObSu} K. Heuer, \textit{A sufficient condition for Hamiltonicity in locally finite graphs}, Europ. J. Combinatorics \textbf{45} (2015) 97--114.

\bibitem{lehner-HC} F. Lehner, \textit{On spanning tree packings of highly edge connected graphs}, J. Combin. Theory Ser. B \textbf{105} (2014) 93--126.

\bibitem{ore} \O. Ore, \textit{Note on Hamilton circuits},  Amer. Math. Monthly \textbf{67} (1960) 55.

\end{thebibliography}
\end{document}